\documentclass[12pt]{amsart}
\usepackage{amsmath}
\usepackage{amsxtra}
\usepackage{amscd}
\usepackage{amsthm}
\usepackage{amsfonts}
\usepackage{amssymb}
\usepackage{eucal}
\usepackage{verbatim}
\usepackage[all]{xy}
\usepackage{color}
\definecolor{deepjunglegreen}{rgb}{0.0, 0.29, 0.29}
\definecolor{darkspringgreen}{rgb}{0.09, 0.45, 0.27}
\definecolor{Red}{rgb}{0.7, 0,0}

\makeatletter
\usepackage{etex,etoolbox,needspace}
\pretocmd\section{\Needspace*{4\baselineskip}}{}{}
\makeatletter

\usepackage[pdftex,bookmarks=false,colorlinks=true,citecolor=darkspringgreen,debug=true,
  naturalnames=true,pdfnewwindow=true]{hyperref}

\newtheorem{thm}{Theorem}[subsection]

\newtheorem{cor}[thm]{Corollary}

\newtheorem{lem}[thm]{Lemma}
\newtheorem{prop}[thm]{Proposition}
\newtheorem{conj}[thm]{Conjecture}

\theoremstyle{definition}

\theoremstyle{remark}

\newcommand{\nc}{\newcommand}
\nc{\renc}{\renewcommand} \nc{\ssec}{\subsection}
\nc{\sssec}{\subsubsection}

\nc{\on}{\operatorname} \nc{\wh}{\widehat}
\nc\ol{\overline} \nc\ul{\underline} \nc\wt{\widetilde}

\newcommand{\red}[1]{{\color{Red}#1}}

\nc{\BA}{{\mathbb{A}}} \nc{\BC}{{\mathbb{C}}} \nc{\BF}{{\mathbb{F}}}
\nc{\BD}{{\mathbb{D}}} \nc{\BG}{{\mathbb{G}}} \nc{\BQ}{{\mathbb{Q}}}
\nc{\BM}{{\mathbb{M}}} \nc{\BN}{{\mathbb{N}}} \nc{\BO}{{\mathbb{O}}}
\nc{\BP}{{\mathbb{P}}} \nc{\BR}{{\mathbb{R}}}
\nc{\BZ}{{\mathbb{Z}}} \nc{\BS}{{\mathbb{S}}} \nc{\BW}{{\mathbb{W}}}

\nc{\CA}{{\mathcal{A}}} \nc{\CB}{{\mathcal{B}}} \nc{\CalC}{{\mathcal{C}}} \nc{\CalD}{{\mathcal{D}}}
\nc{\CE}{{\mathcal{E}}} \nc{\CF}{{\mathcal{F}}}
\nc{\CG}{{\mathcal{G}}} \nc{\CH}{{\mathcal{H}}}
\nc{\CI}{{\mathcal{I}}} \nc{\CK}{{\mathcal{K}}} \nc{\CL}{{\mathcal{L}}}
\nc{\CM}{{\mathcal{M}}} \nc{\CN}{{\mathcal{N}}}
\nc{\CO}{{\mathcal{O}}} \nc{\CP}{{\mathcal{P}}}
\nc{\CQ}{{\mathcal{Q}}} \nc{\CR}{{\mathcal{R}}}
\nc{\CS}{{\mathcal{S}}} \nc{\CT}{{\mathcal{T}}}
\nc{\CU}{{\mathcal{U}}} \nc{\CV}{{\mathcal{V}}}  \nc{\CX}{{\mathcal X}}
\nc{\CY}{{\mathcal Y}} \nc{\CW}{{\mathcal{W}}} \nc{\CZ}{{\mathcal{Z}}}

\nc{\cM}{{\check{\mathcal M}}{}} \nc{\csM}{{\check{\mathcal A}}{}}
\nc{\oM}{{\overset{\circ}{\mathcal M}}{}}
\nc{\obM}{{\overset{\circ}{\mathbf M}}{}}
\nc{\oCA}{{\overset{\circ}{\mathcal A}}{}}
\nc{\obA}{{\overset{\circ}{\mathbf A}}{}}
\nc{\ooM}{{\overset{\circ}{M}}{}}
\nc{\osM}{{\overset{\circ}{\mathsf M}}{}}
\nc{\vM}{{\overset{\bullet}{\mathcal M}}{}}
\nc{\nM}{{\underset{\bullet}{\mathcal M}}{}}
\nc{\oD}{{\overset{\circ}{\mathcal D}}{}}
\nc{\obD}{{\overset{\circ}{\mathbf D}}{}}
\nc{\oA}{{\overset{\circ}{\mathbb A}}{}}
\nc{\op}{{\overset{\bullet}{\mathbf p}}{}}
\nc{\cp}{{\overset{\circ}{\mathbf p}}{}}
\nc{\oU}{{\overset{\bullet}{\mathcal U}}{}}
\nc{\ofZ}{{\overset{\circ}{\mathfrak Z}}{}}

\nc{\fa}{{\mathfrak{a}}} \nc{\fb}{{\mathfrak{b}}}
\nc{\fd}{{\mathfrak{d}}} \nc{\fe}{{\mathfrak{e}}} \nc{\ff}{{\mathfrak{f}}}
\nc{\fg}{{\mathfrak{g}}} \nc{\fgl}{{\mathfrak{gl}}}
\nc{\fh}{{\mathfrak{h}}} \nc{\fri}{{\mathfrak{i}}}
\nc{\fj}{{\mathfrak{j}}} \nc{\fk}{{\mathfrak{k}}} \nc{\fl}{{\mathfrak{l}}}
\nc{\fm}{{\mathfrak{m}}} \nc{\fn}{{\mathfrak{n}}}
\nc{\ft}{{\mathfrak{t}}} \nc{\fu}{{\mathfrak{u}}} \nc{\fv}{{\mathfrak{v}}}
\nc{\fw}{{\mathfrak{w}}} \nc{\fz}{{\mathfrak{z}}}
\nc{\fp}{{\mathfrak{p}}} \nc{\fq}{{\mathfrak{q}}} \nc{\frr}{{\mathfrak{r}}}
\nc{\fs}{{\mathfrak{s}}} \nc{\fsl}{{\mathfrak{sl}}}
\nc{\fso}{{\mathfrak{so}}} \nc{\fsp}{{\mathfrak{sp}}} \nc{\osp}{{\mathfrak{osp}}}
\nc{\hsl}{{\widehat{\mathfrak{sl}}}}
\nc{\hgl}{{\widehat{\mathfrak{gl}}}}
\nc{\hg}{{\widehat{\mathfrak{g}}}}
\nc{\chg}{{\widehat{\mathfrak{g}}}{}^\vee}
\nc{\hn}{{\widehat{\mathfrak{n}}}}
\nc{\chn}{{\widehat{\mathfrak{n}}}{}^\vee}

\nc{\fA}{{\mathfrak{A}}} \nc{\fB}{{\mathfrak{B}}} \nc{\fC}{{\mathfrak{C}}}
\nc{\fD}{{\mathfrak{D}}} \nc{\fE}{{\mathfrak{E}}}
\nc{\fF}{{\mathfrak{F}}} \nc{\fG}{{\mathfrak{G}}} \nc{\fH}{{\mathfrak{H}}}
\nc{\fI}{{\mathfrak{I}}} \nc{\fJ}{{\mathfrak{J}}}
\nc{\fK}{{\mathfrak{K}}} \nc{\fL}{{\mathfrak{L}}}
\nc{\fM}{{\mathfrak{M}}} \nc{\fN}{{\mathfrak{N}}}
\nc{\frP}{{\mathfrak{P}}} \nc{\fQ}{{\mathfrak{Q}}}
\nc{\fS}{{\mathfrak{S}}} \nc{\fT}{{\mathfrak{T}}} \nc{\fU}{{\mathfrak{U}}}
\nc{\fV}{{\mathfrak{V}}} \nc{\fW}{{\mathfrak{W}}}
\nc{\fX}{{\mathfrak{X}}} \nc{\fY}{{\mathfrak{Y}}}
\nc{\fZ}{{\mathfrak{Z}}}

\nc{\ba}{{\mathbf{a}}}
\nc{\bb}{{\mathbf{b}}} \nc{\bc}{{\mathbf{c}}} \nc{\be}{{\mathbf{e}}}
\nc{\bg}{{\mathbf{g}}} \nc{\bj}{{\mathbf{j}}} \nc{\bm}{{\mathbf{m}}}
\nc{\bn}{{\mathbf{n}}} \nc{\bp}{{\mathbf{p}}}
\nc{\bq}{{\mathbf{q}}} \nc{\br}{{\mathbf{r}}} \nc{\bt}{{\mathbf{t}}}
\nc{\bfu}{{\mathbf{u}}} \nc{\bv}{{\mathbf{v}}}
\nc{\bx}{{\mathbf{x}}} \nc{\by}{{\mathbf{y}}} \nc{\bz}{{\mathbf{z}}}
\nc{\bw}{{\mathbf{w}}} \nc{\bA}{{\mathbf{A}}}
\nc{\bB}{{\mathbf{B}}} \nc{\bC}{{\mathbf{C}}}
\nc{\bD}{{\mathbf{D}}} \nc{\bF}{{\mathbf{F}}} \nc{\bG}{{\mathbf{G}}}
\nc{\bH}{{\mathbf{H}}} \nc{\bI}{{\mathbf{I}}} \nc{\bJ}{{\mathbf{J}}}
\nc{\bK}{{\mathbf{K}}} \nc{\bL}{{\mathbf{L}}} \nc{\bM}{{\mathbf{M}}}
\nc{\bN}{{\mathbf{N}}}
\nc{\bO}{{\mathbf{O}}} \nc{\bS}{{\mathbf{S}}} \nc{\bT}{{\mathbf{T}}}
\nc{\bU}{{\mathbf{U}}} \nc{\bV}{{\mathbf{V}}} \nc{\bW}{{\mathbf{W}}}
\nc{\bX}{{\mathbf{X}}}
\nc{\bY}{{\mathbf{Y}}} \nc{\bP}{{\mathbf{P}}}
\nc{\bZ}{{\mathbf{Z}}} \nc{\bh}{{\mathbf{h}}}

\nc{\sA}{{\mathsf{A}}} \nc{\sB}{{\mathsf{B}}}
\nc{\sC}{{\mathsf{C}}} \nc{\sD}{{\mathsf{D}}}
\nc{\sE}{{\mathsf{E}}} \nc{\sF}{{\mathsf{F}}} \nc{\sG}{{\mathsf{G}}} \nc{\sH}{{\mathsf{H}}}
\nc{\sI}{{\mathsf{I}}} \nc{\sK}{{\mathsf{K}}} \nc{\sL}{{\mathsf{L}}}
\nc{\sfm}{{\mathsf{m}}} \nc{\sM}{{\mathsf{M}}} \nc{\sN}{{\mathsf{N}}}
\nc{\sO}{{\mathsf{O}}} \nc{\sQ}{{\mathsf{Q}}} \nc{\sP}{{\mathsf{P}}}
\nc{\sT}{{\mathsf{T}}} \nc{\sZ}{{\mathsf{Z}}}
\nc{\sV}{{\mathsf{V}}} \nc{\sW}{{\mathsf{W}}}
\nc{\sfp}{{\mathsf{p}}} \nc{\sq}{{\mathsf{q}}} \nc{\sr}{{\mathsf{r}}}
\nc{\sfs}{{\mathsf{s}}} \nc{\st}{{\mathsf{t}}} \nc{\sfb}{{\mathsf{b}}}
\nc{\sfc}{{\mathsf{c}}} \nc{\sd}{{\mathsf{d}}}
\nc{\sv}{{\mathsf{v}}}  \nc{\sz}{{\mathsf{z}}}

\nc{\tA}{{\widetilde{\mathbf{A}}}}
\nc{\tB}{{\widetilde{\mathcal{B}}}}
\nc{\tg}{{\widetilde{\mathfrak{g}}}} \nc{\tG}{{\widetilde{G}}}
\nc{\TM}{{\widetilde{\mathbb{M}}}{}}
\nc{\tO}{{\widetilde{\mathsf{O}}}{}}
\nc{\tU}{{\widetilde{\mathfrak{U}}}{}} \nc{\TZ}{{\tilde{Z}}}
\nc{\tx}{{\tilde{x}}} \nc{\tbv}{{\tilde{\bv}}}
\nc{\tfP}{{\widetilde{\mathfrak{P}}}{}} \nc{\tz}{{\tilde{\zeta}}}
\nc{\tmu}{{\tilde{\mu}}}

\nc{\urho}{\underline{\rho}} \nc{\uB}{\underline{B}}
\nc{\uC}{{\underline{\mathbb{C}}}} \nc{\ui}{\underline{i}}
\nc{\uj}{\underline{j}} \nc{\ofP}{{\overline{\mathfrak{P}}}}
\nc{\oB}{{\overline{\mathcal{B}}}}
\nc{\og}{{\overline{\mathfrak{g}}}} \nc{\oI}{{\overline{I}}}

\nc{\eps}{\varepsilon} \nc{\hrho}{{\hat{\rho}}} \nc{\balpha}{{\boldsymbol{\alpha}}}
\nc{\blambda}{{\boldsymbol{\lambda}}} \nc{\bmu}{{\boldsymbol{\mu}}} \nc{\bnu}{{\boldsymbol{\nu}}}
\nc{\btheta}{{\boldsymbol{\theta}}} \nc{\bzeta}{{\boldsymbol{\zeta}}} \nc{\bta}{{\boldsymbol{\eta}}}
\nc{\bbeta}{{\boldsymbol{\beta}}} \nc{\bkappa}{{\boldsymbol{\kappa}}} \nc{\bomega}{{\boldsymbol{\omega}}}

\nc{\one}{{\mathbf{1}}} \nc{\two}{{\mathbf{t}}}

\nc\DMO\DeclareMathOperator
\DMO\Sym{Sym}
\nc{\Tot}{{\mathop{\operatorname{\rm Tot}}}}
\nc{\Spec}{\mathop{\operatorname{\rm Spec}}}
\nc{\Ker}{{\mathop{\operatorname{\rm Ker}}}}
\nc{\Isom}{{\mathop{\operatorname{\rm Isom}}}}
\nc{\Hilb}{{\mathop{\operatorname{\rm Hilb}}}}
\nc{\deeq}{{\mathop{\operatorname{\rm deeq}}}}
\nc{\End}{{\mathop{\operatorname{\rm End}}}}
\nc{\Ext}{{\mathop{\operatorname{\rm Ext}}}}
\nc{\Hom}{{\mathop{\operatorname{\rm Hom}}}}
\nc{\CHom}{{\mathop{\operatorname{{\mathcal{H}}\it om}}}}
\nc{\GL}{{\mathop{\operatorname{\rm GL}}}}
\nc{\PGL}{{\mathop{\operatorname{\rm PGL}}}}
\nc{\SL}{{\mathop{\operatorname{\rm SL}}}}
\nc{\SO}{{\mathop{\operatorname{\rm SO}}}}
\nc{\Sp}{{\mathop{\operatorname{\rm Sp}}}}
\nc{\OSp}{{\mathop{\operatorname{\rm SOSp}}}}
\nc{\gr}{{\mathop{\operatorname{\rm gr}}}}
\nc{\Id}{{\mathop{\operatorname{\rm Id}}}}
\nc{\perf}{{\mathop{\operatorname{\rm perf}}}}
\nc{\defi}{{\mathop{\operatorname{\rm def}}}}
\nc{\length}{{\mathop{\operatorname{\rm length}}}}
\nc{\supp}{{\mathop{\operatorname{\rm supp}}}}
\nc{\HC}{{\mathcal H}{\mathcal C}}
\nc{\Vect}{{\mathop{\operatorname{\rm Vect}}}}

\nc{\pr}{{\operatorname{pr}}}
\nc{\Cliff}{{\mathsf{Cliff}}}
\nc{\loc}{{\operatorname{loc}}} \nc{\lc}{{\operatorname{lc}}}
\nc{\Fl}{{\mathbf{Fl}}} \nc{\Ffl}{{\mathcal{F}\ell}}
\nc{\Fib}{{\mathsf{Fib}}}
\nc{\Coh}{{\mathsf{Coh}}} \nc{\FCoh}{{\mathsf{FCoh}}}
\nc{\Perf}{{\mathsf{Perf}}}
\nc{\wtimes}{\mathbin{\widetilde\times}}

\nc{\reg}{{\text{\rm reg}}} \nc{\ren}{{\text{\rm ren}}}
\nc{\self}{{\text{\rm self}}}
\nc{\gvee}{{\mathfrak g}^{\!\scriptscriptstyle\vee}}
\nc{\tvee}{{\mathfrak t}^{\!\scriptscriptstyle\vee}}
\nc{\nvee}{{\mathfrak n}^{\!\scriptscriptstyle\vee}}
\nc{\bvee}{{\mathfrak b}^{\!\scriptscriptstyle\vee}}
       \nc{\rhovee}{\rho^{\!\scriptscriptstyle\vee}}

\nc{\cplus}{{\mathbf{C}_+}} \nc{\cminus}{{\mathbf{C}_-}}
\nc{\cthree}{{\mathbf{C}_*}} \nc{\Qbar}{{\bar{Q}}}

\newcommand\iso{\mathbin{\vphantom{j^{X^2}}\smash{\overset{\sim}{\vphantom{\rule{0pt}{0.20em}}\smash{\longrightarrow}}}}}
\nc{\Gtimes}{\vphantom{j^{X^2}}\smash{\overset{G}{\vphantom{\rule{0pt}{0.30em}}\smash{\times}}}}
\nc{\sGtimes}{\vphantom{j^{X^2}}\smash{\overset{\mathsf G}{\vphantom{\rule{0pt}{0.30em}}\smash{\times}}}}

\nc{\bOmega}{{\overline{\Omega}}}

\nc{\seq}[1]{\stackrel{#1}{\sim}}

\nc{\aff}{{\operatorname{aff}}}
\nc{\fin}{{\operatorname{fin}}}
\nc{\mir}{{\operatorname{mir}}}
\nc{\triv}{{\operatorname{triv}}}
\nc{\ext}{{\operatorname{ext}}}
\nc{\righ}{{\operatorname{right}}}
\nc{\lef}{{\operatorname{left}}}
\nc{\forg}{{\operatorname{forg}}}
\nc{\fid}{{\operatorname{fd}}}
\nc{\odd}{{\operatorname{odd}}}
\nc{\even}{{\operatorname{even}}}
\nc{\modu}{{\operatorname{-mod}}}
\nc{\Gr}{{\operatorname{Gr}}}
\nc{\FT}{{\operatorname{FT}}}
\nc{\Mat}{{\operatorname{Mat}}}
\nc{\MSt}{{\operatorname{MSt}}}
\nc{\sph}{{\operatorname{sph}}}
\nc{\GR}{{\mathbf{Gr}}}
\nc{\Perv}{{\operatorname{Perv}}}
\nc{\Rep}{{\operatorname{Rep}}}
\nc{\Ind}{{\operatorname{Ind}}}
\nc{\IC}{{\operatorname{IC}}}
\nc{\Bun}{{\operatorname{Bun}}}
\nc{\Proj}{{\operatorname{Proj}}}
\nc{\Stab}{{\operatorname{Stab}}}
\nc{\pt}{{\operatorname{pt}}}
\nc{\bfmu}{{\boldsymbol{\mu}}}
\nc{\bfomega}{{\boldsymbol{\omega}}}
\nc{\calM}{\mathcal M}
\nc{\calA}{\mathcal A}
\nc{\calO}{\mathcal O}
\nc{\CC}{\mathcal C}
\nc{\calN}{\mathcal N}
\nc{\grg}{\mathfrak g}
\nc{\dslash}{/\!\!/}
\nc{\tslash}{/\!\!/\!\!/}
\nc\grt{\mathfrak t}
\nc\bfM{\mathbf M}
\nc\bfN{\mathbf N}
\nc\Sig{\Sigma}
\nc\ZZ{\mathbb{Z}}
\nc\calC{\mathcal C}
\nc\calF{\mathcal F}
\nc\calX{\mathcal X}
\nc\calY{\mathcal Y}
\nc\QCoh{\operatorname{QCoh}}
\nc\IndCoh{\operatorname{IndCoh}}
\nc\Maps{\operatorname{Maps}}
\nc\Dmod{D-\operatorname{mod}}
\newcommand\Hecke{\operatorname{Hecke}}
\nc{\calD}{\mathcal D}
\nc\bfO{\mathbf O}

\nc\GG{\mathbb G}
\nc\calK{\mathcal K}
\nc{\calG}{\mathcal G}
\nc\RHom{\operatorname{RHom}}
\nc\Res{\operatorname{Res}}
\nc\Av{\operatorname{Av}}
\nc{\RH}{{\operatorname{RH}}}
\nc{\RT}{{\operatorname{RT}}}
\nc{\DR}{{\operatorname{DR}}}
\nc{\Ome}{\Omega}
\nc\Lam{\Lambda}

\nc\grs{\mathfrak s}
\nc{\tilX}{\widetilde X}
\nc\calB{\mathcal B}
\nc\calS{\mathcal S}
\nc\calT{\mathcal T}
\nc\calZ{\mathcal Z}
\nc\LS{\operatorname{LocSys}}
\nc\bfL{\on{\mathbf L}}

\newcommand*\circled[1]
{*{#1}}
\makeatletter
\newcommand{\raisemath}[1]{\mathpalette{\raisem@th{#1}}}
\newcommand{\raisem@th}[3]{\raisebox{#1}{$#2#3$}}
\nc{\binlim}[2][]{\def\@tempa{#1}\@ifnextchar^{\@binlim{#2}}{\@binlim{#2}^{}}}
\def\@binlim#1^#2{\mathbin{\@ifempty{#2}{\mathop{#1}}{\mathop{#1}\@xp\displaylimits\@tempa^{#2}}}}

\makeatother

\nc\cX{{\mathcal X}}

\newcommand{\dbkts}[1]{[\![#1]\!]}
\newcommand{\dprts}[1]{(\!(#1)\!)}
\nc\Gm{{\mathbb G_m}}
\nc{\isoto}
    {\stackrel{\displaystyle\raisemath{-.2ex}{\sim}}\to}
\nc\cD\CalD
\nc\setm\setminus
\nc\cE\CE
\renc\Hecke{\mathit{\CH\kern-.2ex ecke}}
\nc\bsl\backslash
\nc\Fq{\mathbb F_q}
\nc\x\times
\nc\bGO{{\bG_\bO}}
\nc\bla\blambda
\nc\blt\bullet
\nc\opp{{\on{op}}}
\nc\srel\stackrel
\nc\oast\circledast
\nc\tbx{\binlim{\widetilde\boxtimes{}}}
\nc\into\hookrightarrow
\nc\otto\leftrightarrow
\renc\setminus\smallsetminus
\nc\phitau{\varphi\tau}
\newenvironment{i-ii-iii}{%
\begin{enumerate}
}%
{\end{enumerate}}
\nc\ceil[1]{\lceil#1\rceil}  \nc\floor[1]{\lfloor#1\rfloor}
\nc\Lie{\on{Lie}}
\nc\sS{{\mathsf S}}
\nc\vvv{\ensuremath{\red\surd}}
\nc\la\lambda
\def\arxiv#1{\href{http://arxiv.org/abs/#1}{\tt arXiv:#1}} \let\arXiv\arxiv
\nc\kap{\kappa}
\nc\gra{\mathfrak a}
\nc\gl{\mathfrak{gl}}
\nc\sTr{\operatorname{sTr}}
\nc\hatG{\widehat{G}}
\nc\calL{\mathcal L}
\nc\Whit{\operatorname{Whit}}
\nc\KL{\operatorname{KL}}

\newcommand\calE{\mathcal E}

\makeatletter
\renewcommand{\subsection}{\@startsection{subsection}{2}{0pt}{-3ex
plus -1ex minus -0.2ex}{-2mm plus -0pt minus
    -2pt}{\normalfont\bfseries}} \makeatother

\nc{\svee}{{\!\scriptscriptstyle\vee}}

\numberwithin{equation}{subsection}
\allowdisplaybreaks
\nc\mto{\mapsto }
\nc\en{\enspace }
\nc\DD{\mathbb{D}}

\begin{document}

\author[A.Braverman]{Alexander Braverman}
\address{Department of Mathematics, University of Toronto and Perimeter Institute
of Theoretical Physics, Waterloo, Ontario, Canada, N2L 2Y5}
\email{braval@math.toronto.edu}

\author[M.Finkelberg]{Michael Finkelberg}
\address{Einstein Institute of Mathematics, The Hebrew University of Jerusalem,
  Edmond J. Safra Campus, Giv’at Ram, Jerusalem, 91904, Israel;
\newline  National Research University Higher School of Economics
}
\email{fnklberg@gmail.com}

\author[D.Kazhdan]{David Kazhdan}
\address{Einstein Institute of Mathematics, The Hebrew University of Jerusalem,
  Edmond J. Safra Campus, Giv’at Ram, Jerusalem, 91904, Israel}
\email{kazhdan@mail.huji.ac.il}


\author[R.Travkin]{Roman Travkin}
\address{Skolkovo Institute of Science and Technology, Moscow, Russia}
\email{roman.travkin2012@gmail.com}

\title
{Relative Langlands duality for $\osp(2n+1|2n)$}




\begin{abstract}
  We establish an $S$-duality converse to the one of~\cite{bft}; this is also a case of a twisted version of the relative Langlands duality of \cite{bzsv}.
  Namely, we prove that the $S$-dual of
  $\SO(2n+1)\times\Sp(2n)\circlearrowright\BC^{2n+1}_+\otimes\BC^{2n}_-$ is
  the symplectic mirabolic space
  $\Sp(2n)\times\Sp(2n)\circlearrowright T^*\Sp(2n)\times\BC^{2n}_-$
  (note that due to the anomaly, the dual of the second factor $\Sp(2n)$ is
  metaplectic dual, i.e.\ $\Sp(2n)$). We also formulate the corresponding global conjecture, which describes explicitly the categorical theta-correspondence on the Langlands dual side.
\end{abstract}

\maketitle

\tableofcontents

\section{Introduction}
\label{intro}
\subsection{Ring objects in the Satake category and relative Langlands duality}
Let $G$ be a connected reductive group over $\BC$. Let $\CK=\BC\dprts t\supset\CO=\BC\dbkts t$. The affine Grassmannian ind-scheme $\Gr_G=G_\CK/G_\CO$
is the moduli space of $G$-bundles on the formal disc equipped with a trivialization on the
punctured formal disc. One can consider the {\em derived Satake category}
$D_{G_\CO}(\Gr_G)$.\footnote{In fact we are going to work with a  renormalized version of it, which by definition is equal to the
ind-completion of the corresponding subcategory of bounded complexes with constructible cohomology.}
This is a factorization monoidal category which is monoidally equivalent to
$D^{G^{\vee}}(\Sym^\bullet(\gvee[-2]))$: the derived category of dg-modules over
$\Sym^\bullet(\gvee[-2])$ endowed with a compatible action of $G^{\vee}$
(the monoidal structure on this category is just given by tensor product over
$\Sym^\bullet(\gvee[-2])$); we shall denote the corresponding functor from $D_{G_\CO}(\Gr_G)$ to
$D^{G^{\vee}}(\Sym^\bullet(\gvee[-2])$ by $\Phi_G$.

In \cite{bfnc} we have attached to any $\bfN$ as above a certain ring object $\calA_{G,\bfN}$ in
$D_{G_\CO}(\Gr_G)$.
This construction was generalized in \cite{bzsv} to the case when $\bfN$ is an arbitrary smooth affine variety with a $G$-action. Moreover, it is argued in \cite{bfnc}, \cite{bdfrt} and \cite{bzsv} that the object $\calA_{G,\bfN}$ should only depend on $\bfM=T^*\bfN$. One of the main conjectures of \cite{bzsv} says that if $\bfN$ is a spherical $G$-variety  then $H^*(\Phi_G(\calA_{G,\bfN})$ should be the algebra of functions on certain ({\em relative Langlands dual} or {\em S-dual}) hyper-spherical Poisson $G$-variety $\bfM^{\vee}$, and this construction is expected to be involutive in some reasonable generality, cf.\ also~\cite{n2}.

\subsection{The non-cotangent case} The construction of $\calA_{G,\bfN}$ should in principle make sense for any smooth affine symplectic $G$-variety $\bfM$.
However, when $\bfM$ is not of cotangent type some extra care is needed. This is discussed in detail in \cite{bdfrt} when $\bfM$ is a symplectic linear representation of $G$. In this case in {\em loc. cit.} the corresponding object $\calA_{G,\bfM}$ was constructed, but in general it is not an object of $D_{G_\CO}(\Gr_G)$ but rather of some twisted version of it. The twisting is by a square root of some line bundle on $\Gr_G$ which depends on $\bfM$; in the case when the twisting is indeed non-trivial we say that an {\em anomaly} is present. The derived Satake equivalence can be extended to such twisted categories (but one has to change the notion of the Langlands dual group $G^{\vee}$) so even in the anomalous case all of the above constructions go through.

\subsection{The subject of~\cite{bft}} This note is a sequel to~\cite{bft}. There we considered the
symplectic group $\Sp(2n)\circlearrowright\BC^{2n}_-$ and its Langlands dual
group $\SO(2n+1)\circlearrowright\BC^{2n+1}_+$. We also considered a symplectic
vector space $\bM=\BC^{2n+1}_+\otimes\BC^{2n}_-$ and the corresponding Weyl
algebra $\CW$ of $\bM_\CK$. Here $\CK=\BC(\!(t)\!)\supset\BC[\![t]\!]=\CO$.
We studied the category $\CalD\CW\modu^{\Sp(2n)_\CO,\lc}$ of locally compact
$\Sp(2n)_\CO$-equivariant objects in the tensor product cagegory
$\on{D-mod}_{1/2}(\Gr_{\Sp(2n)})\otimes\CW\modu$ ($D$-modules on the affine
Grassmannian of $\Sp(2n)$ twisted by the square root of the determinant line
bundle).

We established the following algebraic description of
$\CalD\CW\modu^{\Sp(2n)_\CO,\lc}$ conjectured by D.~Gaiotto. We consider the
(infinite-dimensional) graded algebra $\Sym^\bullet(\varPi\bM[-1])$ ($\varPi$
assigns to $\bM$ the odd parity) as a dg-superalgebra with trivial differential.
We constructed an equivalence of categories
\[\CalD\CW\modu^{\Sp(2n)_\CO,\lc}\iso D_\perf^{\SO(2n+1)\times\Sp(2n)}(\Sym^\bullet(\varPi\bM[-1])).\]
This is also a particular case of the general~\cite[Conjecture 7.5.1]{bzsv}.
In other words, it means that the $S$-dual of the symplectic variety
$T^*\Sp(2n)\times\BC^{2n}_-$ equipped with the hamiltonian action of
$\Sp(2n)\times\Sp(2n)$, is $\bM\circlearrowleft\SO(2n+1)\times\Sp(2n)$.
(Note that due to the twisting by the square root of the determinant line
bundle, the second factor $\Sp(2n)$ corresponds under the $S$-duality to its
metaplectic Langlands dual $\Sp(2n)$.)

\subsection{The subject of this paper}
In the present note we confirm the converse claim: the $S$-dual of
$\bM\circlearrowleft\SO(2n+1)\times\Sp(2n)$ is
$T^*\Sp(2n)\times\BC^{2n}_-\circlearrowleft\Sp(2n)\times\Sp(2n)$.
That is, we construct an equivalence of categories
\begin{multline*}
  D_\perf^{\Sp(2n)\times\Sp(2n)}\left(\BC[\Sp(2n)]\otimes\Sym^\bullet(\fsp(2n)[-2])\otimes\Sym^\bullet(\varPi(\BC^{2n}_-)[-1])\right)\\
  \iso\CW\modu^{\SO(2n+1)_\CO\times\Sp(2n)_\CO,\lc}.
  \end{multline*}

\subsection{The global conjecture}
Another very important point of~\cite{bzsv} is that the local relative Langlands duality is expected to give rise to certain global geometric (``period") representation of automorphic $L$-functions. The precise formulation of~\cite{bzsv} is in the case when the spectral side is of cotangent type and when there is no anomaly. Hopefully, both assumptions can be overcome. We are not going to discuss how to do this in general, however, we are going to present a conjecture in the above case, which we would like to think of as a natural extension of the setting of~\cite[Section 12]{bzsv} to our setting.

Let $C$ be a smooth projective irreducible curve over $\BC$.\footnote{A similar discussion should make sense in the $\ell$-adic setting when $C$ is a curve over a finite field; however in that case we do not know how to formulate a precise conjecture.}
For an algebraic group $G$ we denote by $\Bun_G$ the moduli stack of $G$-bundles on $C$; in the case when $G=\Sp(2n)$ we shall also consider the twisted version of $\Bun_{\Sp(2n)}$ which we shall denote by $\Bun_{\Sp(2n)}^\omega$ --- it classifies vector bundles $M$ on $C$ of rank $2n$ equipped with a non-degenerate skew-symmetric form $\Lam^2(M)\to\omega_C$.
We have a natural morphism $\iota\colon\Bun_{\Sp(2n)}^\omega\times \Bun_{\SO(2n+1)}\to \Bun_{\Sp(2n(2n+1))}^\omega$.
Let $\Theta$ denote the theta-sheaf of~\cite{Ly} on $\Bun_{\Sp(2n(2n+1))}^\omega$. This is actually a sheaf twisted by
the square root of the natural determinant bundle on $\Bun_{\Sp(2n(2n+1))}^\omega$. Consider now $\iota^!\Theta$.
This is a sheaf on $\Bun_{\Sp(2n)}^\omega\times \Bun_{\SO(2n+1)}$ twisted by the square root of the determinant
line bundle along the first factor.

It is expected (cf.~\cite{GaiLys}) that the twisted global geometric Langlands duality should assign to any sheaf $\calF$ on  $\Bun_{\Sp(2n)}^\omega\times \Bun_{\SO(2n+1)}$ twisted by the square root of the determinant bundle along the first factor an ind-coherent sheaf ${\mathbb L}(\calF)$ (with nilpotent singular support) on $\LS_{\Sp(2n)}(C)\times \LS_{\Sp(2n)}(C)$ where for an algebraic group $H$ over $\BC$ we denote by $\LS_H(C)$ the moduli stack of de Rham $H$-local systems on $C$.\footnote{In the untwisted case (in the de Rham setting) the global geometric Langlands conjecture has been recently proved, but the twisted case so far remains completely open.}

To formulate our conjecture we need to introduce the following notation. Let $\calE$ be a symplectic local system on $C$. For simplicity let us assume that $H^0_{\on{dR}}(C,\calE)=H^2_{\on{dR}}(C,\calE)=0$ where the subscript dR stands for de Rham cohomology.  The space $H^1_{\on{dR}}(C,\calE)$ has a canonical symmetric non-degenerate bilinear form; it also has a canonical maximal isotropic subspace $L_\calE$ which is equal to the image of $H^0(C,\calE\otimes \Omega_C)$ in $H^1_{\on{dR}}(C,\calE)$. We set $S_\calE$ to be the corresponding spinor representation of the Clifford algebra $\on{Cliff}(H^1_{\on{dR}}(C,\calE))$ of $H^1_{\on{dR}}(C,\calE)$ (by definition, it is induced from the trivial representation of $\Lambda(L_\calE))$ which is naturally a subalgebra of $\on{Cliff}(H^1_{\on{dR}}(C,\calE))$).

The following conjecture appears in a slightly weaker form in \cite[Conjecture 1.2.4]{Ly2020} for $n=1$:
\begin{conj}
\begin{enumerate}
\item
There exists a sheaf $\Theta^{\vee}$ on $\LS_{\Sp(2n)}(C)$ such that ${\mathbb L}(\iota^!\Theta)$ is equal to $\Delta_*\Theta^{\vee}$ where $\Delta\colon \LS_{\Sp(2n)}(C)\to\LS_{\Sp(2n)}(C)\times \LS_{\Sp(2n)}(C)$ is the diagonal embedding.
\item The fiber of $\Theta^{\vee}$ at $\calE$ as above is equal to $S_\calE$.

\end{enumerate}

\end{conj}

\begin{cor}
  Let $\Psi$ denote the natural functor from $D(\Bun_{\SO(2n+1)})$ to  $D_{-1/2}(\Bun_{\Sp(2n)}^\omega)$ defined by the kernel
  $\iota^!\Theta$ (here $D_{-1/2}(\Bun_{\Sp(2n)}^{\omega})$ stands for the corresponding twisted category of $D$-modules).
  Let $\Psi^*$ denote the functor in the opposite direction given by $\DD(\iota^!\Theta)$ (Verdier duality).
  Let also $\calE$ be as above and let $A_\calE$ be a Hecke eigen-D-module on $\Bun_{\SO(2n+1)})$ with eigenvalue $\calE$.
  Then $\Psi^*\circ\Psi(A_\calE)\simeq A_{\calE}\otimes \on{Cliff}(H^1_{\on{dR}}(C,\calE))$.
\end{cor}
Note that $\on{Cliff}(H^1_{\on{dR}}(C,\calE))$ can be thought of as a categorification of the value of the $L$-function of $\calE$ at $1/2$, so the above corollary is a version of the correspondence between period sheaves and $L$-sheaves which is the main subject of \cite{bzsv}.

\subsection{Acknowledgments}
We are grateful to A.~Bouthier, D.~Gaiotto, N.~Gurevich, S.~Lysenko and H.~Nakajima for the inspiring discussions.
We are also indebted to D.~Mezer for the careful reading of the manuscript and spotting a lot of inaccuracies.
Our note follows in the footsteps of an unpublished work by T.-H.~Chen
and J.~Wang, cf.~\cite{cw}.

The research of A.B.\ was partially supported by NSERC; in addition it was started when A.B.\ was visiting
the University of Geneva, and he thanks this institution for its hospitality.
The research of M.F.\ was supported by the Israel Science Foundation (grant No.~994/24).
The research of D.K.\ was partially supported by ERC grant 101142781.

\section{Setup}

\subsection{Weyl algebra}
Let $V=\BC^{2n}$ be a symplectic vector space; the symplectic pairing is
denoted by $\langle,\rangle$. Let $V'=\BC^{2n+1}$ be a vector space equipped
with a nondegenerate symmetric pairing $(,)$. We set $\bM:=V'\otimes V$
and equip it with the tensor product symplectic form, also do be denoted
by $\langle,\rangle$.

The symplectic form on $\bM$ extends to the same named $\BC$-valued symplectic form on
$\bM_\CK\colon \langle f,g\rangle=\on{Res}\langle f,g\rangle_\CK dt$. We denote by $\CW$ the completion
of the Weyl algebra of $(\bM_\CK,\langle\, ,\rangle)$ with respect to the left ideals generated by
the compact subspaces of $\bM_\CK$. It has an irreducible representation $\BC[\bM_\CO]$.

We consider the dg-category $\CW\modu$ of discrete $\CW$-modules. We recall this is a renormalization of the naive derived category $\CW\modu^{\on{naive}}$ of discrete $\CW$ modules, or more carefully its canonical dg-enhancement, defined as follows.

For each compact open subspace $\mathbf{K} \subset \mathbf{M}_{\CK}$, consider the module $U_{\mathbf{K}}$ obtained as the quotient of $\CW$ by the left ideal generated by $\mathbf{K}$. Let us denote by $\mathcal{E}$ the pre-triangulated envelope of all such modules $U_{\mathbf{K}}$ within $\CW\modu^{\on{naive}}$. By definition, $\CW\modu$ is the ind-completion of $\mathcal{E}$. It carries a unique $t$-structure for which the natural map
\[\CW\modu \rightarrow \CW\modu^{\on{naive}}\]
is $t$-exact.

 More concretely, we may identify
$\CW$ with the ring of differential operators on a Lagrangian discrete lattice $\bL\subset\bM_\CK$,
e.g.\ $\bL=t^{-1}\bM_{\BC[t^{-1}]}$. Then $\CW\modu$ is the inverse limit of ${\rm D}\modu(U)$ over
finite dimensional subspaces $U\subset\bL$ with respect to the functors $i_{U\hookrightarrow U'}^!$.
Equivalently, $\CW\modu$ is the colimit, in the sense of cocomplete dg-categories, of ${\rm D}\modu(U)$ with respect to the functors
$i_{U\hookrightarrow U',*}$. The following lemma is a consequence of~\cite[\S10]{r},
see~\cite[Lemma 2.4.1]{bdfrt}.

\begin{lem}
  There is a categorical action \[\on{D-mod}_{-1/2}(\Sp(\bM)_\CK)\circlearrowright\CW\modu.\]
  In particular, upon taking spherical vectors, there is an action \[\on{D-mod}_{-1/2}(\Gr_{\Sp(\bM)})^{\Sp(\bM)_\CO}\circlearrowright(\CW\modu)^{\Sp(\bM)_\CO}.\]
\end{lem}

We have an embedding $\SO(V')_\CO\times\Sp(V)_\CO\hookrightarrow\Sp(\bM)_\CO$,
and we denote by $\CW\modu^{\SO(V')_\CO\times\Sp(V)_\CO}$ the category of
$\SO(V')_\CO\times\Sp(V)_\CO$-invariants in $\CW\modu$. Furthermore, we denote
by $\CW\modu^{\SO(V')_\CO\times\Sp(V)_\CO,\lc}$ the category of locally compact
objects of $\CW\modu^{\SO(V')_\CO\times\Sp(V)_\CO}$. (Recall that an equivariant
object is called locally compact if it becomes a compact object of $\CW\modu$
after forgetting the equivariant structure.)

\subsection{Main theorem}
\label{main}
We consider a dg-superalgebra with trivial differential
$\fA^\bullet:=\BC[\Sp(V)]\otimes\Sym^\bullet(\fsp(V)[-2])\otimes\Sym^\bullet(\varPi(V)[-1])$
(here $\varPi(V)$ stands for the vector space $V$ made odd, so that
$\Sym^\bullet(\varPi(V)[-1])$ is infinite-dimensional).
It is equipped with the following action of $\Sp(V)\times\Sp(V)\colon
(g_1,g_2)(g,x,v)=(g_1gg_2^{-1},\on{Ad}_{g_1}x,g_1v)$.
We also consider a dg-algebra with trivial differential
$\fG^\bullet:=\Sym^\bullet(\fsp(V)[-2])\otimes\Sym^\bullet(\varPi(V)[-1])$
equipped with the following action of $\Sp(V)\colon
g(x,v)=(\on{Ad}_gx,gv)$.
We have an obvious equivalence of categories
$D_\perf^{\Sp(V)\times\Sp(V)}(\fA^\bullet)\cong D_\perf^{\Sp(V)}(\fG^\bullet)$
(equivariant perfect dg-modules).

Our goal is the following

\begin{thm}
  \label{main thm}
  There is an equivalence of triangulated categories
  $\Phi\colon D_\perf^{\Sp(V)\times\Sp(V)}(\fA^\bullet)\iso
  \CW\modu^{\SO(V')_\CO\times\Sp(V)_\CO,\lc}$ commuting with the convolution
  action of the monoidal spherical Hecke category\footnote{For the equivalence
  of first factors, see~\cite{bf}, and for the equivalence of
  second factors, see~\cite{dlyz}.}
  \begin{multline*}
    D^{\Sp(V)}_\perf(\Sym^\bullet(\fsp(V)[-2]))\otimes
    D^{\Sp(V)}_\perf(\Sym^\bullet(\fsp(V)[-2]))\\
    \cong
    \on{D-mod}(\Gr_{\SO(V')})^{\SO(V')_\CO}\otimes
    \on{D-mod}_{-1/2}(\Gr_{\Sp(V)})^{\Sp(V)_\CO}.
  \end{multline*}
\end{thm}

The proof occupies the rest of the paper.

\section{The proof}

\subsection{The two Hecke actions can be identified}
\label{two actions}
We identify $\CW$ with differential operators on $(V'\otimes V)_\CK/(V'\otimes V)_\CO$, and consider a $D$-module $E_0$: the delta-function
of the origin. Then $E_0\in\CW\on{-mod}^{\SO(V')_\CO\times\Sp(V)_\CO}$ is the unit object.
We denote by $\phi_{\Sp}$ (resp.\ $\phi_{\SO}$) the geometric Satake equivalence
$\on{Rep}(\Sp(V))\iso\on{D-mod}_{-1/2}(\Gr_{\Sp(V)})^{\Sp(V)_\CO,\heartsuit}$
(resp.\ $\on{Rep}(\Sp(V))\iso\on{D-mod}(\Gr_{\SO(V')})^{\SO(V')_\CO,\heartsuit}$)
(hearts of the natural $t$-structures).

\begin{lem}
  \label{two act}
  Under the convolution actions of $\on{D-mod}(\Gr_{\SO(V')})^{\SO(V')_\CO,\heartsuit}$ and
  $\on{D-mod}_{-1/2}(\Gr_{\Sp(V)})^{\Sp(V)_\CO,\heartsuit}$ on $\CW\on{-mod}^{\SO(V')_\CO\times\Sp(V)_\CO}$,
  for any representation $U$ of $\Sp(V)$, the convolution actions $\phi_{\SO}(U)*E_0$ and $\phi_{\Sp}(U)*E_0$,
  coincide (compatibly with the tensor structures).
\end{lem}

\begin{proof}    
  It suffices to

  a) identify the two convolution actions for the generator $V$ of
  $\on{Rep}(\Sp(V))$;

  b) under the above identification to identify the direct summands
  \begin{multline*}
    E_0=\phi_{\Sp}(\BC)*E_0\subset\phi_{\Sp}(V)*\phi_{\Sp}(V)*E_0\cong\phi_{\Sp}(V)*\phi_{\SO}(V)*E_0\\
    \cong\phi_{\SO}(V)*\phi_{\Sp}(V)*E_0\cong\phi_{\SO}(V)*\phi_{\SO}(V)*E_0\supset\phi_{\SO}(\BC)*E_0=E_0,
    \end{multline*}
corresponding to the embedding of the direct summand $\BC\subset V\otimes V$ in $\Rep(\Sp(V))$.

a) We view $\CW\on{-mod}$ as $D$-modules on
  $\bM_\CK/\bM_\CO$, then both $\phi_{\SO}(V)*E_0$ and $\phi_{\Sp}(V)*E_0$ are
  supported on $(t^{-1}\bM_\CO)/\bM_\CO$.
  Here is a description of these $D$-modules, cf.~\cite[\S2.13]{bft}. Let
  $Q\subset\Gr_{\SO(V')}$ denote the minuscule $\SO(V')_\CO$-orbit, and let
  $\Gr_{\Sp(V)}^{\on{min}}\subset\Gr_{\Sp(V)}$ denote the minimal $\Sp(V)_\CO$-orbit.
  Recall that $\Gr_{\Sp(V)}^{\on{min}}$ is the total space of the line bundle $\CO_{\BP(V)}(2)$ over $\BP(V)$. In other words,
  $\Gr_{\Sp(V)}^{\on{min}}$ is the moduli space of pairs $(\ell,h)$, where $\ell$ is a line in $V$, and
  $h\in\Hom(\ell,V/\ell^\perp)$.
  Namely, to $(\ell,h)$ we associate a Lagrangian lattice $L\subset V_\CK$ equal to $tV_\CO\oplus\ell^\perp\oplus\Gamma_h$, where
  $\Gamma_h\subset V_\CO/(\ell^\perp\oplus tV_\CO)\oplus t^{-1}\ell = (V/\ell^\perp)\oplus t^{-1}\ell$ is the graph of $th$. Now
  we consider the total space $X$ of the pullback of $V'\otimes\CO_{\BP(V)}(-1)$ to $\Gr_{\Sp(V)}^{\on{min}}$. It carries a function
  $f$ such that $f(\ell,h,v'\otimes v)=(v',v')\cdot\langle v,hv\rangle$ (here $v\in\ell$).
  We denote by $\CF_{\on{min}}$ the $D$-module $\exp(f)$ on $X$. 

  Furthermore, $Q\subset\Gr_{\SO(V')}$ is the space of self-orthogonal lattices
  $L'=tV'_\CO\oplus \ell^{\prime\perp}\oplus t^{-1}\ell'$, where $\ell'\subset V'$
  is an isotropic line. We consider the total space $X'$ of $\CO_Q(-1)\otimes V$. We denote by $\CF_Q$ the
  constant $D$-module on $X'$. Finally, we denote by $\pr_2$ the natural projections of $X$ and $X'$ to $(t^{-1}\bM_\CO)/\bM_\CO$. 
  Then $\phi_\SO(V)*E_0=\pr_{2*}\CF_Q$, and $\phi_\Sp(V)*E_0=\pr_{2*}\CF_{\on{min}}$.
  (Note that the extension of $\CF_{\on{min}}$ to
  $\overline\Gr{}^{\on{min}}_{\Sp(V)}\times(t^{-1}\bM_\CO)/\bM_\CO$ is clean.)

  Clearly, $\CF_Q$ is a regular $D$-module, and $\phi_\SO(V)*E_0=\pr_{2*}\CF_Q$ is a regular $D$-module
  as well. Moreover, $\phi_\Sp(V)*E_0=\pr_{2*}\CF_{\on{min}}$ is also regular.
  Indeed, $\pr_{2*}\CF_{\on{min}}=\pr_*p_*\CF_{\on{min}}$, where
  \[p\colon\Gr_{\Sp(V)}^{\on{min}}\times(t^{-1}\bM_\CO)/\bM_\CO\to
  \BP((t^{-1}V_\CO)/V_\CO)\times(t^{-1}\bM_\CO)/\bM_\CO\] is the projection
  (contraction by the loop rotation action) in the first factor, and
  $\pr\colon \BP((t^{-1}V_\CO)/V_\CO)\times(t^{-1}\bM_\CO)/\bM_\CO\to
  (t^{-1}\bM_\CO)/\bM_\CO$ is the second projection. But $p_*\CF_{\on{min}}$
  is already a regular $D$-module supported on $\{(\ell,t^{-1}V'\otimes\ell)\}$,
  i.e.\ on the total space of $t^{-1}V'\otimes\CO_{\BP(V)}(-1)$.

  Now that we know that both $\phi_{\SO}(V)*E_0$ and $\phi_{\Sp}(V)*E_0$ are
  regular, we describe the corresponding constructible (perverse) sheaves.

  Let $C\subset V'\otimes V\cong t^{-1}(V'\otimes V)_\CO/(V'\otimes V)_\CO$ be the cone formed
  by all the tensors of the form $v'\otimes v$, where $v'\in V'_0$ (the cone of isotropic vectors).
  Then $C$ has two small resolutions: $((V'_0\setminus\{0\})\times V))/\BC^\times_{\on{hyperb}}$, and
  $(V'_0\times(V\setminus\{0\}))/\BC^\times_{\on{hyperb}}$ (note that the second ``resolution'' is only {\em rationally} smooth). The pushforwards of the constant sheaves
  coincide with $\phi_{\SO}(V)*E_0$ and $\phi_{\Sp}(V)*E_0$ respectively; and they both coincide
  with the Goresky-MacPherson sheaf of~$C$. This completes the identification in~a).

  b) We have $\Hom(E_0,\phi_{\Sp}(V)*\phi_{\Sp}(V)*E_0)=\Hom(\phi_{\Sp}(V)*E_0,\phi_{\Sp}(V)*E_0)=
  \Hom(\on{IC}_C,\on{IC}_C)=\BC$. So the direct summand $E_0$ in $\phi_{\Sp}(V)*\phi_{\Sp}(V)*E_0$
  is uniquely determined. Similarly, the direct summand $E_0$ in $\phi_{\SO}(V)*\phi_{\SO}(V)*E_0$
  is uniquely determined and must coincide with the former summand under the identification
  in~a).
\end{proof}

\subsection{The Hecke action generates $\CW\on{-mod}^{\SO(V')_\CO\times\Sp(V)_\CO}$}
We first classify irreducible $\GL(N)_\CO\times\GL(N)_\CO$-equivariant $D$-modules on
$\on{Mat}(N\times N,\CK)$. Here by $D$-modules we mean the ones with support in a lattice in
$\on{Mat}(N\times N,\CK)$ invariant under translations by a smaller lattice. Thus we need to
classify the $\GL(N)_\CO\times\GL(N)_\CO$-orbits in $\on{Mat}(N\times N,\CK)$ invariant under
translations by a lattice. These orbits are exactly the ones in $\GL(N)_\CK$, and they are
indexed by the length $N$ signatures $\blambda$.

Note that $\BO_\bmu$ lies in $\overline\BO_\blambda$
iff \[\lambda_1\geq\mu_1,\ \lambda_1+\lambda_2\geq\mu_1+\mu_2,\ \ldots,\ \lambda_1+\cdots+\lambda_N
\geq\mu_1+\cdots+\mu_N.\] The last condition is different from the usual definition of dominating
order, where we require $\lambda_1+\cdots+\lambda_N
=\mu_1+\cdots+\mu_N$. Indeed, $\on{Mat}(N\times N,\CK)$ is connected, while $\pi_0(\GL(N)_\CK)=\BZ$.

Similarly, one proves
\begin{lem}
  The $\GL(M)_\CO\times\GL(N)_\CO$-orbits in $\on{Mat}(M\times N,\CK)$ invariant under translations
  by a sublattice are indexed by the set of length $M$ signatures (we assume $M\leq N$).
  An orbit $\BO_\blambda$ has a representative $(t^{-\blambda},0)$ (the last $N-M$ columns are all zero).
\end{lem}

The Fourier transform gives an equivalence
\[\on{D-mod}(\on{Mat}(M\times N,\CK))^{\GL(M)_\CO\times\GL(N)_\CO}
\iso\on{D-mod}(\on{Mat}(N\times M,\CK))^{\GL(N)_\CO\times\GL(M)_\CO}.\] The irreducible objects in
both sides are indexed by the (translation invariant) orbits of $\GL(N)_\CO\times\GL(M)_\CO$,
i.e.\ by the length $M$
signatures. Hence we obtain an involution $\on{FT}$ of the set of length $M$ signatures.
\begin{lem}
  $\on{FT}(\lambda_1\geq\ldots\geq\lambda_M)=(-\lambda_M\geq\ldots\geq-\lambda_1)$
\end{lem}

\begin{proof}
  The Fourier transform is compatible with the action of $\GL(N)_\CK\times\GL(M)_\CK$, and hence
  is compatible with the Hecke action. But duality induces the Chevalley involution on the Hecke
  category.
\end{proof}

Now we take $M=2n$, $N=2n+1$, and realize $\SO(V')\times\Sp(V)$ as the (connected component of the)
fixed point set of an appropriate involution on $\GL(N)\times\GL(M)$. The argument
of~\cite[\S2.6]{bft} establishes

\begin{cor}
  \label{irr index}
  The irreducible objects of $\CW\on{-mod}^{\SO(V')_\CO\times\Sp(V)_\CO,\heartsuit}$ are indexed by
  the cone of dominant weights of $\Sp(V)$. Hence $\CW\on{-mod}^{\SO(V')_\CO\times\Sp(V)_\CO}$ is
  generated by a collection od objects $\phi_\SO(U)*E_0$, $U\in\on{Rep}(\Sp(V))$.
\end{cor}

\subsection{A deequivariantized Ext-algebra}
\label{deeq}
Making use of the convolution action
$\Rep(\Sp(V))\circlearrowright\CW\on{-mod}^{\SO(V')_\CO\times\Sp(V)_\CO}$, we
consider the deequivariantized category
$\CW\on{-mod}^{\SO(V')_\CO\times\Sp(V)_\CO}_\deeq$. Recall the unit object
$E_0$ in $\CW\on{-mod}^{\SO(V')_\CO\times\Sp(V)_\CO}$ introduced in~\S\ref{two actions}.
We will keep the same notation $E_0$ for the corresponding object in
$\CW\on{-mod}^{\SO(V')_\CO\times\Sp(V)_\CO}_\deeq$.

The following lemma is proved the same way as~\cite[Lemma 2.7.1]{bft}.

\begin{lem}
  \label{formal}
  The dg-algebra $\RHom_{\CW\on{-mod}^{\SO(V')_\CO\times\Sp(V)_\CO}_\deeq}(E_0,E_0)$ is
  formal, i.e.\ it is quasiisomorphic to the graded algebra
  $\Ext^\bullet_{\CW\on{-mod}^{\SO(V')_\CO\times\Sp(V)_\CO}_\deeq}(E_0,E_0)$ with trivial
  differential.
  \end{lem}
We denote the dg-algebra
$\Ext^\bullet_{\CW\on{-mod}^{\SO(V')_\CO\times\Sp(V)_\CO}_\deeq}(E_0,E_0)$ (with trivial
differential) by $\fE^\bullet$. Since it is an Ext-algebra in the deequivariantized
category between objects induced from the original category, it is automatically
equipped with an action of $\Sp(V)$, and we can consider the corresponding
triangulated category $D^{\Sp(V)}_\perf(\fE^\bullet)$. The following lemma is
proved the same way as~\cite[Lemma 2.7.2]{bft}, making use
of~Corollary~\ref{irr index}.

\begin{lem}
  \label{equiva}
  There is a canonical equivalence
  $D^{\Sp(V)}_\perf(\fE^\bullet)\iso\CW\on{-mod}^{\SO(V')_\CO\times\Sp(V)_\CO}$.
\end{lem}

We choose a decomposition $V=L\oplus L^*$ of $V$ into a direct sum of two
transversal Lagrangian subspaces. It defines a Siegel Levi subgroup
$\GL(L)\subset\Sp(V)$. We choose Kostant slices $\Sigma_\fgl\subset\fgl(L)$
and $\Sigma_\fsp\subset\fsp(V)$. Making use of the trace forms,
we identify
$H^\bullet_{\GL(L)_\CO}(\pt)\cong\BC[\Sigma_\fgl]$ and
$H^\bullet_{\Sp(V)_\CO}(\pt)\cong\BC[\Sigma_\fsp]\cong H^\bullet_{\SO(V')_\CO}(\pt)$.

Restricting the equivariance from $\Sp(V)_\CO$ to $\GL(L)_\CO$, we obtain
the deequivariantized Ext-algebra
$\widetilde\fE^\bullet:=\Ext^\bullet_{\CW\on{-mod}^{\SO(V')_\CO\times\GL(L)_\CO}_\deeq}(E_0,E_0)$.

The purity argument used in the proofs of Lemmas~\ref{formal} and~\ref{equiva}
also establishes the following

\begin{lem}
  \label{free}
  \textup{(1)} $\fE^\bullet$ is a free module over
  $H^\bullet_{\SO(V')_\CO\times\Sp(V)_\CO}(\pt)\cong\BC[\Sigma_\fsp\times\Sigma_\fsp]$.

  \textup{(2)} $\widetilde\fE^\bullet\cong\BC[\Sigma_\fsp\times\Sigma_\fgl]
  \otimes_{\BC[\Sigma_\fsp\times\Sigma_\fsp]}\fE^\bullet$.
  \end{lem}

Now we can prove

\begin{lem}
  \label{commut}
  The algebras $\fE^\bullet,\widetilde\fE^\bullet$ are commutative.
\end{lem}

\begin{proof}
  Making use of the $\SO(V')\times\GL(L)$-invariant decomposition
  $V'\otimes V=V'\otimes L\oplus V'\otimes L^*$, we identify the categories
  \[\CW\on{-mod}^{\SO(V')_\CO\times\GL(L)_\CO}\cong
  \on{D-mod}((V'\otimes L)_\CK)^{\SO(V')_\CO\times\GL(L)_\CO}.\] The action of the
  central $\BG_m\subset\GL(L)\subset\GL(L)_\CO$ contracts $(V'\otimes L)_\CK$
  to the origin. Hence by the Localization Theorem we have an isomorphism
  $\widetilde\fE^\bullet_\loc\simeq\Ext^\bullet_{\on{D-mod}(\pt)^{\SO(V')_\CO\times\GL(L)_\CO}_\deeq}(\delta_0,\delta_0)_\loc$.
  Here for a module $M$ over
  $H^\bullet_{\SO(V')_\CO\times\GL(L)_\CO}(\pt)\cong\BC[\Sigma_\fsp\times\Sigma_\fgl]$
  we denote by $M_\loc$ the extension of scalars to the field of fractions.
  More precisely, we apply the Localization Theorem to the cohomology
  equivariant with respect to a Cartan torus, and then take the Weyl group
  invariants. The localization isomorphism is compatible with multiplication since for a
  $\SO(V')_\CO\times\GL(L)_\CO$-equivariant $D$-module on $(V'\otimes L)_\CK$, the restriction to
  the origin coincides with the equivariant cohomology, and the latter is compatible with the action
  of $\Rep(\Sp(V))$ via Satake equivalence, like
  in~Lemma~\ref{tensor functor} below.
  
  Now obviously, $\Ext^\bullet_{\on{D-mod}(\pt)^{\SO(V')_\CO\times\GL(L)_\CO}_\deeq}
  (\delta_0,\delta_0)_\loc\simeq\BC[\Sp(V)\times\Sigma_\fsp\times\Sigma_\fgl]_\loc$
  is commutative. By~Lemma~\ref{free}, $\widetilde\fE^\bullet$ embeds into
  $\widetilde\fE^\bullet_\loc$. Hence $\widetilde\fE^\bullet$ is commutative,
  and $\fE^\bullet$ is commutative as well.
\end{proof}

Recall that $\fG^\bullet:=\Sym^\bullet(\fsp(V)[-2])\otimes\Sym^\bullet(\varPi(V)[-1])$ (see~\S\ref{main}).
We construct a homomorphism $\psi\colon\fG^\bullet\to\fE^\bullet$ as follows.
First, the derived Satake equivalence $D^{\Sp(V)}_\perf(\Sym^\bullet(\fsp(V)[-2]))
\cong\on{D-mod}(\Gr_{\SO(V')})^{\SO(V')_\CO}$ along with the convolution action
$\on{D-mod}(\Gr_{\SO(V')})^{\SO(V')_\CO}\circlearrowright\CW\on{-mod}^{\SO(V')_\CO\times\Sp(V)_\CO}$
gives rise to a homomorphism $\Sym^\bullet(\fsp(V)[-2])\to\fE^\bullet$.
Second, the computation of $\phi_\SO(V)*E_0$ in the proof of~Lemma~\ref{two act}
produces a canonical nonzero element of cohomological degree~1 in
$\iota_0^!(\phi_\SO(V)*E_0)$. Here $\iota_0$ stands for the embedding of
the origin into the colattice $(V'\otimes V)_\CK/(V'\otimes V)_\CO$, and
$E_0=\delta_0$ as a $D$-module over $(V'\otimes V)_\CK/(V'\otimes V)_\CO$.
Since $\Ext^\bullet_{\CW\on{-mod}^{\SO(V')_\CO\times\Sp(V)_\CO}}(E_0,\phi_{\SO}(V)*E_0\otimes V^*)$
is a direct summand of $\fE^\bullet$ (and we have $V^*\cong V$), we obtain the
desired embedding $V[-1]\hookrightarrow\fE^\bullet$.

By~Lemma~\ref{equiva}, in order to establish the equivalence
of~Theorem~\ref{main thm}, it remains to prove the following

\begin{prop}
  \label{psi}
  The homomorphism $\psi\colon\fG^\bullet\to\fE^\bullet$ is an isomorphism.
\end{prop}

Its proof will be given in~\S\ref{eq coh} after some preparation from Invariant Theory.

\subsection{A pseudo-slice}
\label{pseudo}
We choose two transversal Lagrangian subspaces $V=L\oplus L^*$ as in~\S\ref{deeq}.
We choose a regular $\fsl_2$-triple $(e,h,f)$ in $\fsp(V)$ such that $L=\on{Ker}(f^n)$,
and $L^*=\on{Ker}(e^n)$. We choose a vector
$v^*_0\in L^*\cap eL$ such that $\langle v_0^*,fv_0^*\rangle=n$.
The Kostant slice $\Sigma_\fsp=e+\fz_{\fsp(V)}(f)$.
We define $\Sigma:=\Sigma_\fsp\times(v^*_0+L)\subset\fsp(V)\oplus V$. We denote by $\sigma$ the embedding
$\Sigma\hookrightarrow\fsp(V)\oplus V$.

According to~\cite[row 3 of Table 4a]{s}, $\BC[\fsp(V)\oplus V]^{\Sp(V)}$ is a polynomial algebra
with $2n$ generators $\on{Tr}(x^2),\ldots,\on{Tr}x^{2n},\langle v,xv\rangle,\ldots,
\langle v,x^{2n-1}v\rangle$.
Note that the projection $\pi\colon\Sigma\to(\fsp(V)\oplus V)/\!\!/\Sp(V)
\cong(\fsp(V))/\!\!/\Sp(V)\times(\fsp(V))/\!\!/\Sp(V)\cong\Sigma_\fsp\times\Sigma_\fsp\simeq\BA^{2n}$,
$(x,v)\mapsto[x],[x-v^2]$, is a surjective (ramified) cover of degree $2^n$.
This cover is {\em not} Galois for $n>1$. In fact, for any
$n\geq1$, the projection
$\fsp(V)\oplus V\twoheadrightarrow(\fsp(V)\oplus V)/\!\!/\Sp(V)\cong
\Sigma_\fsp\times\Sigma_\fsp$ admits {\em no} Weierstra{\ss} section.

Note that the Lagrangian subspace $L$ is invariant with respect to $x-v^2$ for $(x,v)\in\Sigma$.
We consider the composed map $\Sigma\to\fgl(L)\to\fgl(L)/\!\!/\GL(L)=:\Sigma_\fgl$, $(x,v)\mapsto\on{Spec}(x-v^2)$.
Together with the first projection $\Sigma\to\Sigma_\fsp$, $(x,v)\mapsto x$,
we obtain a morphism $\varpi\colon\Sigma\to\Sigma_\fsp\times\Sigma_\fgl$.

\begin{lem}
  The map $\varpi\colon\Sigma\to\Sigma_\fsp\times\Sigma_\fgl$ is an isomorphism.
\end{lem}

\begin{proof}
The map $\varpi$ takes $(x,v^*_0+\ell)\in\Sigma$ to $(x,\on{Spec}(\on{pr}_L(x)+v^*_0\otimes \ell))$,
where $\on{pr}_L\colon\fsp(V)\to\fgl(L)$ is the projection to the $L$-block with respect to
decomposition $V=L\oplus L^*$. We set
$v_i:=f^{n+1-i}v_0^*\in\on{Ker}(f^i)\setminus\on{Ker}(f^{i-1})$, and set
$v=v^*_0+\sum_{i=1}^n a_iv_i$. Let $t^n+\sum_{i=1}^nc_it^{n-i}$ be the characteristic polynomial of
$\on{pr}_L(x)+v^*_0\otimes \ell$. Then $c_i=\varkappa_ia_i+\varphi_i$, where
$\varkappa_i\ne0$ is a constant, while $\varphi_i$ is a
function of $a_j$, $1\leq j\langle i$, and coordinates on $\Sigma_\fsp$.
\end{proof}

\begin{lem}\label{Zcodim2}
  The saturation $\Sp(V)\cdot\Sigma\subset\fsp(V)\oplus V$ of $\Sigma$ is a constructible
  subset of $\fsp(V)\oplus V$. It contains an open subset $\CU$ such that
  \textup{(1)} the complement $\CZ=\fsp(V)\oplus V\setminus\CU$ has codimension~2 in $\fsp(V)\oplus V$;
  \textup{(2)} $\CU$ contains the full preimage of a nonempty open subset in $\Sigma_\fsp\times\Sigma_\fsp$.
\end{lem}

\begin{proof}
  The image of the morphism $\Sp(V)\times\Sigma\to\fsp(V)\times V\times\on{LGr}(V)$
  (Lagrangian Grassmannian), $(g,x,v)\mapsto(\on{Ad}_g(x),gv,gL)$,
  contains a locally closed subvariety
  $X\subset\fsp(V)\times V\times\on{LGr}(V)$ consisting of triples $(x',v',L')$ satisfying the
  following conditions:

  a) $x'\in\fsp(V)$ is a regular element,

  b) $x'-v^{\prime2}\in\fsp(V)$ preserves $L'$,

  c) $v'\pmod{L'}$ is a cyclic vector for $x'-v^{\prime2}\pmod{L'}$.

  \noindent Indeed, let us choose a cyclic vector $\sv$ for the endomorphism
  $e\colon V\to V$ such that $v_0^*=e^n\sv$, and $L^*=\BC e^{2n-1}\sv\oplus\BC e^{2n-2}\sv\oplus\ldots\oplus\BC e^n\sv$,
  while $L=\BC e^{n-1}\sv\oplus\ldots\oplus\BC e\sv\oplus\BC\sv$.
  Then given a triple $(x',v',L')$ satisfying (a--c) above, we can find $g\in\Sp(V)$ taking $L'$ to $L$, and
  $v'$ to $v_0^*$. Now let $\pr\colon V\to L^*$ be the projection along $L$. We consider the following basis of
  of $L^*$: $\{\pr(x'-v^{\prime2})^{n-1}v_0^*,\pr(x'-v^{\prime2})^{n-2}v_0^*,\ldots,v_0^*\}$. We take its union with the
  dual basis of $L$, and obtain a basis of $V$. In this basis, the matrix of
  $x'-v^{\prime2}$ has both the upper left block and the lower right block with all nonzero elements right above
  the diagonal and all zeros above them.
  Conjugating by an appropriate element $g'\in\GL(L)\subset\Sp(V)$, we can make all the matrix
  elements of $g'(x'-v^{\prime2})g^{\prime-1}$ in the basis $\{e^{2n-1}\sv,\ldots,e^n\sv,e^{n-1}\sv,\ldots,\sv\}$ of $V$, 
  right above the diagonals in both diagonal blocks to be exactly~1. Now $x'=(x'-v^{\prime2})+v^{\prime2}$ has all $2n-1$ elements
  right above the diagonal equal to~1, and all zeros above them. In other words, $x'\in e+\fb$, where $\fb$ is
  the Borel subalgebra of $\fsp(V)$ containing $f$. Let $\fu$ be the nilpotent radical of $\fb$, and let
  $\on{U}\subset\Sp(V)$ be the corresponding unipotent subgroup. We define $\CA\subset\fsp(V)\oplus V$ as
  $\CA=\{e+\fb,v_0^*+L\}$. Then the map $\on{U}\times\Sigma\to\CA$, $(u,x,v)\mapsto u\cdot\sigma(x,v)$ is an
  isomorphism. Hence we can find $u\in U$ that takes $(x',v')$ to $\Sigma$. Clearly, $\on{U}$ stabilizes $L$.
  Thus $X$ is contained in the image of $\Sp(V)\times\Sigma$, as desired.

  Furthermore, the natural morphism $\Sp(V)\times\Sigma\to\fsp(V)\oplus V$,
  $(g,x,v)\mapsto(\on{Ad}_g(x)-(gv)^2,gv)$, factors as the
  composition $\Sp(V)\times\Sigma\to\fsp(V)\times V\times\on{LGr}(V)\to\fsp(V)\oplus V$, where
  the second morphism takes $(x',v',L')$ to $(x'-v^{\prime2},v')$.

  So we must find the desired open subspace $\CU$ in the image of $X$ in $\fsp(V)\oplus V$.
  First we take care of the condition~(1).
If $x'-v^{\prime2}$ has a zero eigenvalue (a codimension~1 condition), then for $(x',v')$ in the image
of $X$, the corresponding component $v'_0$ is a cyclic vector for the corresponding block
$(x'-v^{\prime2})_0$. The complement to the set of such pairs has codimension~1. All in all, the
set of pairs $(x',v')$ such that $x'-v^{\prime2}$ is not invertible, and $v'_0$ is not a cyclic vector
for $(x'-v^{\prime2})_0$ has codimension~2 in $\fsp(V)\oplus V$.

Now if $x'-v^{\prime2}$ is invertible, for each pair of opposite eigenvalues $(\lambda,-\lambda)$,
either the corresponding component $v'_\lambda$ is a cyclic vector for the corresponding block
$(x'-v^{\prime2})_\lambda$, or $v'_{-\lambda}$ is a cyclic vector for $(x'-v^{\prime2})_{-\lambda}$.
The complement to the set of such pairs again has codimension~2 in $\fsp(V)\oplus V$.

As for the condition~(2), $\CU$ contains the open subset formed by all $(x',v')$ such that
all the eigenvalues of $x'$ are distinct from the eigenvalues of $x'-v^{\prime2}$, and all the eigenvalues
of $x'$ and of $x'-v^{\prime2}$ have multiplicity~1.

This completes the proof of the lemma.
\end{proof}

We consider the fibre product (over the second copy of $\Sigma_\fsp=\fsp(V)/\!\!/\Sp(V)$)
$\fM:=(\fsp(V)\oplus V)\times_{\Sigma_\fsp}\Sigma_\fgl$. The above maps $\Sigma\to\fsp(V)\oplus V$
and $\Sigma\to\Sigma_\fgl$ give rise to a closed embedding $\Sigma\hookrightarrow\fM$.

\begin{cor}
  \label{Weier}
  \textup{(1)} We have $\fM/\!\!/\Sp(V)=\Sigma_\fsp\times\Sigma_\fgl$, and $\Sigma\hookrightarrow\fM$ is a Weierstra\ss\
  section of the $\Sp(V)$-action.

  \textup{(2)}
  $\BC[\fsp(V)\oplus V] = \BC[\Sigma\times\Sp(V)] \times_{\BC(\fM)} \BC(\fsp(V)\oplus V)$.
\end{cor}

\subsection{Equivariant cohomology}
\label{eq coh}
Recall from the proof of~Lemma~\ref{Zcodim2} that $\Sigma$ is equipped with
a $\BG_m$-action. It corresponds to the half-grading in the RHS of the
isomorphism $\BC[\Sigma]\cong H^\bullet_{\SO(V')_\CO\times\GL(L)_\CO}(\pt)$.
We have a functor $\chi\colon\Rep(\Sp(V))\to\Vect^{\BG_m}(\Sigma)$
($\BG_m$-equivariant vector bundles),
$U\mapsto H^\bullet_{\SO(V')_\CO\times\GL(L)_\CO}((V'\otimes L)_\CK,\phi_\SO(U)*E_0)$.
Here we view $\phi_\SO(U)*E_0$ as an $\SO(V')_\CO\times\GL(L)_\CO$-equivariant
$D$-module on $(V'\otimes L)_\CK$.

\begin{lem}
  \label{tensor functor}
  $\chi$ is a tensor functor.
\end{lem}

\begin{proof}
  We equip $\CW\on{-mod}^{\SO(V')_\CO\times\GL(L)_\CO}\cong
  \on{D-mod}((V'\otimes L)_\CK)^{\SO(V')_\CO\times\GL(L)_\CO}$ with the fusion
  tensor structure. Then the functor
  \[\Rep(\Sp(V))\to\on{D-mod}((V'\otimes L)_\CK)^{\SO(V')_\CO\times\GL(L)_\CO},\
  U\mapsto\phi_\SO(U)*E_0\] is a tensor functor. Its essential image
  lies in the full subcategory
  $\on{D-mod}((V'\otimes L)_\CK)^{\SO(V')_\CO\times\GL(L)_\CO}_{\on{ss}}$
  of semisimple $D$-modules in the heart of the
  natural $t$-structure. The functor
  $\on{D-mod}((V'\otimes L)_\CK)^{\SO(V')_\CO\times\GL(L)_\CO}_{\on{ss}}\to
  \Vect^{\BG_m}(\Sigma)$,
  $\CF\mapsto H^\bullet_{\SO(V')_\CO\times\GL(L)_\CO}((V'\otimes L)_\CK,\CF)$
  is also a tensor functor. Indeed, consider the dilation action
  of $\BG_m$ on $(V'\otimes L)_\CK$. It coincides with the action of the central
  $\BG_m\subset\GL(L)\subset\SO(V')\times\GL(L)\subset\SO(V')_\CO\times\GL(L)_\CO$.
  It contracts $(V'\otimes L)_\CK$ to the origin, and $\CF$ is a monodromic
  $D$-module. Hence $H^\bullet_{\SO(V')_\CO\times\GL(L)_\CO}((V'\otimes L)_\CK,\CF)\cong i_0^*\CF$,
  where $i_0$ stands for the embedding of the origin into $(V'\otimes L)_\CK$.
  So the desired tensor property
  is the simplest instance of the commutation of
  nearby cycles and hyperbolic restriction~\cite[Proposition 5.4.1(2)]{n}.

      Finally, the tensor property of $\chi$ follows since it is a composition
      of two tensor functors.
\end{proof}

The tensor functor $\chi$ defines a $\BG_m$-equivariant $\Sp(V)$-torsor $\CT$
over $\Sigma$. The deequivariantized Ext-algebra $\fE^\bullet$ by definition
acts on $\BC[\Sp(V)]*E_0$, and we obtain a morphism $\CT\to\Spec\fE$. In other
words, we obtain a morphism $\alpha\colon\Sigma\to\Spec\fE/\Sp(V)$
(the stacky quotient).

The homomorphism $\psi\colon\fG^\bullet\to\fE^\bullet$ of~Proposition~\ref{psi}
gives rise to the same named $\BG_m\times\Sp(V)$-equivariant morphism of
spectra $\Spec(\fE)\to\Spec(\fG)=\fsp(V)\oplus V$. Consider the composition
\[\Sigma\xrightarrow{\alpha}\Spec(\fE)/\Sp(V)\xrightarrow{\psi}(\fsp(V)\oplus V)/\Sp(V).\]

\begin{lem}
  \label{calcul}
The composition
\(\Sigma\xrightarrow{\alpha}\on{Spec}\fE/\Sp(V)\xrightarrow{\psi}(\fsp(V)\oplus V)/\Sp(V)
\to(\fsp(V)\oplus V)/\!\!/\Sp(V)\)
is the tautological map $\Sigma\to\Sigma_\fsp\times\Sigma_\fsp$.
\end{lem}

\begin{proof}
We have three maps from $\fsp(V)[-2]$ to $\fE^\bullet$. The first, $\psi_1$ arises from 
the derived Satake equivalence $D^{\Sp(V)}_\perf(\Sym^\bullet(\fsp(V)[-2]))
\cong\on{D-mod}(\Gr_{\SO(V')})^{\SO(V')_\CO}$ along with the convolution action
$\on{D-mod}(\Gr_{\SO(V')})^{\SO(V')_\CO}\circlearrowright\CW\on{-mod}^{\SO(V')_\CO\times\Sp(V)_\CO}$.
It also makes use of the invariant identification $\fsp(V)\simeq\fsp(V)^*$, and it will be convenient
for us to choose this identification as the {\em negative} minimal invariant symmetric bilinear form.
Similarly, the second map $\psi_2$ arises from
the derived Satake equivalence $D^{\Sp(V)}_\perf(\Sym^\bullet(\fsp(V)[-2]))
\cong\on{D-mod}_{-1/2}(\Gr_{\Sp(V)})^{\Sp(V)_\CO}$ along with the convolution action
$\on{D-mod}_{-1/2}(\Gr_{\Sp(V)})^{\Sp(V)_\CO}\circlearrowright\CW\on{-mod}^{\SO(V')_\CO\times\Sp(V)_\CO}$
(and the minimal invariant symmetric bilinear form on $\fsp(V)$).
The third, $\psi_3$ takes a linear function on $\fsp(V)^*$, restricts it to a quadratic function on $V^*$ (via the
minimal nilpotent coadjoint orbit), and then sends it to the square (in $\fE^\bullet$) of the map
$\varpi(V)[-1]\to\fE^\bullet$ described before~Proposition~\ref{psi}. We have to check that $\psi_1-\psi_2=\psi_3$.

Consider the object $E_\omega:=\phi_\SO(V)*E_0\cong\phi_\Sp(V)*E_0\in\CW\on{-mod}^{\SO(V')_\CO\times\Sp(V)_\CO}$.
Then $\psi_i$ can be viewed as elements in $\Ext^2_{\CW\on{-mod}^{\SO(V')_\CO\times\Sp(V)_\CO}}(E_\omega,E_\omega)$.
Namely, $\psi_1$, $\psi_2$ arise from certain elements
$\varepsilon_1\in\Ext^2_{\on{D-mod}(\Gr_{\SO(V')})^{\SO(V')_\CO}}(\phi_\SO(V),\phi_\SO(V))$,
$\varepsilon_2\in\Ext^2_{\on{D-mod}_{-1/2}(\Gr_{\Sp(V)})^{\Sp(V)_\CO}}(\phi_\Sp(V),\phi_\Sp(V))$.
More precisely, $\varepsilon_1$ is the multiplication by the {\em negative} first Chern class
$-c_1(\CalD_\SO)$ of the determinant line bundle on $\Gr_{\SO(V')}$, and $\varepsilon_2$ is the
multiplication by half the first Chern class $\frac12 c_1(\CalD_\Sp)$ of the determinant line bundle
on $\Gr_{\Sp(V)}$ (one half comes from the twisting of $D$-modules on $\Gr_{\Sp(V)}$).
Finally, $\psi_3=\epsilon\circ\epsilon'$, where $\epsilon\in\Ext^1_{\CW\on{-mod}^{\SO(V')_\CO\times\Sp(V)_\CO}}(E_0,E_\omega)$
is the canonical element described before~Proposition~\ref{psi}, and
$\epsilon'\in\Ext^1_{\CW\on{-mod}^{\SO(V')_\CO\times\Sp(V)_\CO}}(E_\omega,E_0)$ is a similar canonical element.

Recall that $E_\omega=\IC(C)$ was described in the proof of~Lemma~\ref{two act} as the pushforward from any of the
small ``resolutions'' $\widetilde{C}_1=((V'_0\setminus\{0\})\times V)/\BC^\times_{\on{hyperb}}$ and
$\widetilde{C}_2=(V'_0\times(V\setminus\{0\}))/\BC^\times_{\on{hyperb}}$. By~\cite[Lemma 8.6.1]{cg} we can identify
$\Ext^2_{\CW\on{-mod}^{\SO(V')_\CO\times\Sp(V)_\CO}}(E_\omega,E_\omega)\cong
H_{\on{top}-2}^{\SO(V')\times\Sp(V)}(\widetilde{C}_1\times_C\widetilde{C}_2)$ (Borel-Moore homology of degree $8n-4$).
Note that $H_{\on{top}-2}^{\SO(V')\times\Sp(V)}(\widetilde{C}_1\times_C\widetilde{C}_2)\cong
H^2_{\SO(V')\times\Sp(V)}(\widetilde{C}_1\times_C\widetilde{C}_2)$.
Now $\widetilde{C}_1\times_C\widetilde{C}_2$ is the total space of the line bundle
$\CO_Q(-1)\boxtimes\CO_{\BP^{2n-1}}(-1)$ over $Q\times\BP^{2n-1}$. The fundamental class of the zero section
$Q\times\BP^{2n-1}\subset\widetilde{C}_1\times_C\widetilde{C}_2$ is equal to
$c_1(\CO_Q(-1)\boxtimes\CO_{\BP^{2n-1}}(-1))=-c_1(\CO_Q(1))-c_1(\CO_{\BP^{2n-1}}(1))=
-c_1(\CalD_\SO|_Q)-\frac12 c_1(\CalD_\Sp|_{\Gr^{\on{min}}_{\Sp(V)}})=\varepsilon_1-\varepsilon_2$.
Finally, the composition $\epsilon\circ\epsilon'$ is equal to the product of fundamental classes of
$Q$ and $\BP^{2n-1}$, by definition of convolution in the Borel-Moore homology.
\end{proof}

\begin{lem}
  \label{tors}
  \textup{(1)} The composition $\psi\circ\alpha$ coincides (up to a unique isomorphism) with the tautological morphism
  $\Sigma\to(\fsp(V)\oplus V)/\Sp(V)$.

  \textup{(2)} The $\Sp(V)$-torsor $\CT\to\Sigma$ is trivialized:
  $\CT\cong\Sp(V)\times\Sigma$.
\end{lem}

\begin{proof}
(1) We know that the further composition
  $\Sigma\to\on{Spec}\fE/\Sp(V)\to\on{Spec}\fG/\Sp(V)\to(\fsp(V)\oplus V)/\!\!/\Sp(V)$
  is the tautological map $\Sigma\to\Sigma_\fsp\times\Sigma_\fsp$ by~Lemma~\ref{calcul}. We also know
  that $\psi\circ\alpha\colon\Sigma\to(\fsp(V)\oplus V)/\Sp(V)$ is equivariant
  with respect to the $\BG_m$-action on $\fsp(V)\oplus V$ given by
  $(2\rho^\svee-2,2\rho^\svee-1)$ (that is
  $c(x,v)=(c^{-2}\on{Ad}_{2\rho^\svee(c)}x,c^{-1}(2\rho^\svee(c))v)$).

  So we have to check that there is a unique $\BG_m$-equivariant morphism
  $\gamma\colon\Sigma\to(\fsp(V)\oplus V)/\Sp(V)$ whose composition with the map
  $(\fsp(V)\oplus V)/\Sp(V)\to(\fsp(V)\oplus V)/\!\!/\Sp(V)$ is the
  tautological map $\Sigma\to\Sigma_\fsp\times\Sigma_\fsp$. By~\cite[Theorem 3]{w}, the
  $\Sp(V)$-torsor $\CT\to\Sigma$ can be $\BG_m$-equivariantly trivialized. Hence $\gamma$
  can be lifted to a $\BG_m$-equivariant map $\tilde\gamma\colon\Sigma\to\fsp(V)\oplus V$.

  First we check that the base point $0\in \Sigma$ must go to (the orbit of)
  the point $(e,v_0^*)\in\fsp(V)\oplus V$. Indeed, the first component of
  $\tilde\gamma(0)$ must be a regular nilpotent in $\fsp(V)$: otherwise the
  composition of $\gamma$ with the first projection
  $(\fsp(V)\oplus V)/\Sp(V)\to\fsp(V)/\!\!/\Sp(V)=\Sigma_\fsp$ will not be
  smooth at $0\in\Sigma$ (the differential will not be surjective).
  So we can assume that the first component of $\tilde\gamma(0)$ is $e\in\fsp(V)$.
  The second component of $\tilde\gamma(0)$ must be a $\BG_m$-fixed point in $V$,
  so it is $cv_0^*$ for a constant $c$, and we must prove that $c=\pm1$.

Now consider the Borel subalgebra $\fb\subset\fsp(V)$ containing $f$.
  Let $\fu$ be the nilpotent radical of $\fb$, and let $\on{U}\subset\Sp(V)$
  be the corresponding unipotent subgroup.
  The attractor $\CA_c$ to the point $(e,cv_0^*)$ of the $\BG_m$-action on $\fsp(V)\oplus V$
  is $\{(e+\fb,cv_0^*+L)\}$. Let $\sigma_c\colon\Sigma\hookrightarrow\fsp(V)\oplus V$ be the embedding
  $\Sigma_\fsp\times(v_0^*+L)\ni(x,v_0^*+\ell)\mapsto(x,cv_0^*+\ell)$. Then the map
  $\on{U}\times\Sigma\to\CA_c$, $(u,x,v)\mapsto u\cdot\sigma_c(x,v)$ is an isomorphism.
  Hence given a $\BG_m$-equivariant morphism $\tilde\gamma\colon\Sigma\to\CA_c$, there is
  a morphism $\Upsilon\colon\Sigma\to\on{U}$ and an endomorphism
  $\alpha\colon\Sigma\to\Sigma$ such that $\tilde\gamma=\on{Ad}_\Upsilon(\sigma_c\circ\alpha)$.
  Furthermore, since the tautological cover $\Sigma\to\Sigma_\fsp\times\Sigma_\fsp$ admits no section,
  and we assume that the composition of $\tilde\gamma$
  with the projection $\CA_c\to(\fsp(V)\oplus V)/\!\!/\Sp(V)=\Sigma_\fsp\times\Sigma_\fsp$
  is the tautological map $\Sigma\to\Sigma_\fsp\times\Sigma_\fsp$, it follows that $\alpha$ must be an
  automorphism of $\Sigma$ (preserving the $\BG_m$ fixed point).

  Finally, one can check that the tautological cover $\Sigma\to\Sigma_\fsp\times\Sigma_\fsp$ is
  ramified over $\Sigma_\fsp\times\{0\}$. However, if $c\ne\pm1$, then the composition
  \(\Sigma\xrightarrow{\sigma_c}\fsp(V)\oplus V\twoheadrightarrow
  (\fsp(V)\oplus V)/\!\!/\Sp(V)=\Sigma_\fsp\times\Sigma_\fsp\) is {\em not}
  ramified over $\Sigma_\fsp\times\{0\}$. We conclude that $c=\pm1$, and moreover, we can take $c=1$
  (composing with $-1\in\Sp(V)$ if there be need), so that $\sigma_c=\sigma$ is the canonical canonical
  embedding $\Sigma\hookrightarrow\fsp(V)\oplus V$, and $\alpha=\Id_\Sigma$.

   This completes the proof of (1).

  (2) follows since the tautological morphism $\Sigma\to(\fsp(V)\oplus V)/\Sp(V)$
  arises from the following map of the trivial $\Sp(V)$-torsor $\Sp(V)\times\Sigma$ into
  $\fsp(V)\oplus V$, $(g,x,v)\mapsto g\cdot\sigma(x,v)$.
\end{proof}

Now we can prove~Proposition~\ref{psi}. Recall the algebras
$\widetilde\fE\subset\widetilde\fE_\loc$ introduced in~\S\ref{deeq}.
(We ignore the gradings from now on, so we omit the superscript $^\bullet$.)
Similarly, we introduce an algebra $\widetilde\fG:=\BC[\fM]$
(recall from~\S\ref{pseudo} that $\fM:=(\fsp(V)\oplus V)\times_{\Sigma_\fsp}\Sigma_\fgl$, and $\fG=\BC[\fsp(V)\oplus V]$).
Then the homomorphism $\psi$ of~Proposition~\ref{psi} induces
$\tilde\psi\colon \widetilde\fG\to\widetilde\fE$.
From the proof of~Lemma~\ref{commut} we know that $\tilde\psi$ is injective,
since both $\widetilde\fG$ and $\widetilde\fE$ embed into
$\widetilde\fG_\loc\simeq\widetilde\fE_\loc$. It follows that $\psi$ is
an embedding $\fG\hookrightarrow\fE\subset\on{Frac}(\fG)=\BC(\fsp(V)\oplus V)$.
Hence we may view an element $f\in\fE$ as a rational function on
$\fsp(V)\oplus V$. We know that the pullback of this rational function to
$\Sp(V)\times\Sigma$ (under the natural morphism
$\Sp(V)\times\Sigma\to\fsp(V)\oplus V$, $(g,x,v)\mapsto g\cdot\sigma(x,v)$)
is regular by~Lemma~\ref{tors}. By~Corollary~\ref{Weier}(2) we conclude
that $f$ is regular. Hence the embedding $\psi\colon\fG\hookrightarrow\fE$
is an equality.

This completes the proof of~Proposition~\ref{psi} along with~Theorem~\ref{main thm}.

\end{document}